\documentclass[12pt]{amsart}
\usepackage{amsmath,amssymb,txfonts}
\usepackage{amssymb}
\usepackage{amsxtra}
\usepackage{amsthm, color}
\usepackage{txfonts}
\usepackage{graphicx}
\usepackage{times}

\usepackage{hyperref}
\allowdisplaybreaks
\usepackage{color}
\usepackage{xcolor}
\usepackage{pgf,tikz,pgfplots}
\usepackage{mathrsfs}
\usetikzlibrary{arrows}

\hypersetup{
        colorlinks   = true,
        citecolor    = blue,
        linkcolor    = blue,
        urlcolor     = blue
}

\makeatletter
\@namedef{subjclassname@1991}{Subject}
\@namedef{subjclassname@2000}{Subject}
\@namedef{subjclassname@2010}{Subject}
\@namedef{subjclassname@2020}{{\rm 2020} Mathematics Subject Classification}
\makeatother

\def\rr{{\mathbb R}}
\def\rn{{{\rr}^n}}

\def\nn{{\mathbb N}}
\def\zz{{\mathbb Z}}

\def\ls{\lesssim}
\def\gs{\gtrsim}

\def\r{\right}
\def\lf{\left}

\def\XXint#1#2#3{{\setbox0=\hbox{$#1{#2#3}{\int}$ }
\vcenter{\hbox{$#2#3$ }}\kern-.6\wd0}}

      \newcommand{\lt}[1]{[ {#1}] \lower.3ex\hbox{$_{t}$}}

\date{\today}

\newtheorem{theorem}{Theorem}[section]
\newtheorem{definition}[theorem]{Definition}

\newtheorem{remark}[theorem]{Remark}

\newtheorem{proposition}[theorem]{Proposition}
\newtheorem{lemma}[theorem]{Lemma}

\begin{document}
\title[Boundedness of commutator]
 {Boundedness of commutator generated by fractional integral operator and Orlicz-BMO function}
	
\author{Zixing Zhuang}
\address{School of Mathematics and Statistics,
                       Fuzhou University,
                       Fuzhou, Fujian 350108, People's Republic of China}
\email{Zhuangzxmath@126.com}

\author{Chenglong Fang$^\ast$}
\address{School of Mathematics and Statistics,
                       Fuzhou University,
                       Fuzhou, Fujian 350108, People's Republic of China}
\email{clfang@fzu.edu.cn}

\author{Liwen Cao}
\address{School of Mathematics and Statistics,
         Fuzhou University,
         Fuzhou, Fujian 350108, People's Republic of China}
\email{Caolwmath@126.com}

\thanks{Chenglong Fang is supported by the National Natural Science Foundation of China (\# 12501119) and the Fujian Provincial Natural Science Foundation Young Scientist Innovation Project (\# 2025J08035).}

\thanks{$^\ast$ Corresponding author}

	\date{}
		
	

\begin{abstract}
For $\alpha\in(0, n)$ and a growth function $\varphi:[0,\infty)\rightarrow [0,\infty)$,
it is proved that the commutator $[b,I_\alpha]$ generated by fractional integral operator $I_\alpha$ and Orlicz $\mathrm{BMO}$ function $b$ is bounded from Orlicz-Hardy space $H_{b}^{\varphi}(\mathbb{R}^{n})$ to Lebesgue space $L^{\frac{n}{n-\alpha}}(\mathbb{R}^{n})$,
where $H_{b}^{\varphi}(\mathbb{R}^{n})$ is a suitable Orlicz-Hardy space. Moreover,
the authors also establish that
the boundedness of commutator $[b,I_\alpha]$
from Orlicz-Hardy space $H^{\varphi}(\mathbb{R}^{n})$ to weak Lebesgue space $L^{\frac{n}{n-\alpha},\infty}(\mathbb{R}^{n})$.
\end{abstract}

\maketitle	

%

\section{Introduction}

For $0<\alpha<n$, the fractional integral operator $I_\alpha$ is defined as follows:
 $$I_\alpha (g)(x):=\int_{\mathbb{R}^{n}}\frac{g(z)}{\lf|x-z\r|^{n-\alpha}}\,dz \quad \text{for all}\, x \in \rn.$$
Two classical results (see \cite[Theorem~1.2.3]{GTM250} and \cite[Theorem~4]{A11}) asserted that there exists constants $C$ such that
\begin{align}\label{eq-Ln-L1}
\|I_\alpha g\|_{L^{\frac{n}{n-\alpha}, \, {\infty}}(\mathbb{R}^n)}
\leq C \|g\|_{L^{1}(\rn)}
\end{align}
and
\begin{align*}
\|I_\alpha g\|_{L^{q}(\mathbb{R}^n)}
\leq C \|g\|_{L^{p}(\rn)}
\end{align*}
with $1<p<q<\infty$ and $1/q-1/p=\alpha/n$.
In the framework of Hardy spaces, the boundedness of $I_\alpha$ is considered by Stein and Weiss \cite[Theorem~G]{A13} in their pioneering work on multi-variable harmonic functions, they established that
\begin{align}\label{eq-H-1}
\|I_\alpha g\|_{L^{\frac{n}{n-\alpha}}(\mathbb{R}^n)}
\leq C \|g\|_{H^{1}(\rn)},
\end{align}
which reveals the compatibility between the regularity of fractional integral operator and the atomic structure of Hardy spaces.

Given a locally integrable function $b$ and a sublinear operator $T$,
the commutator is denoted by $[b,T]$, it is given by
$$[b,T]g:= bTg-T(bg),$$
where $g$ is a continuous function whose support is compact.
A famous theorem of Chanillo \cite{C1982} characterizes the $\mathrm{BMO}(\rn)$ space via the boundedness of $[b,I_\alpha]$, that is,
$b\in \mathrm{BMO}(\rn)$ if and only if the commutator $[b,I_\alpha]$ is bounded from $L^{p}(\rn)$ to $L^{q}(\rn)$, where $0<\alpha<n$ and $1<p<\frac{n}{\alpha}$ with
$1/q=1/p-\alpha/n$.
Paluszy\'{n}ski \cite[Theorem~1.1]{Palus1995} proved that $[b,I_\alpha]$ does not map  from $H^1(\rn)$ into $L^{\frac{n}{n-\alpha}}(\rn)$ continuously.
Later,
P\'{e}rez \cite{P1995} introduced a subspace $\mathcal{H}^{1}_{b}(\rn)\subset H^{1}(\rn)$ in the study of commutator of Calder\'on-Zygmund operator.
Prompted by the bilinear splitting technique applied to the pair of Hardy space $H^1(\rn)$ and $\mathrm{BMO}$ space, which is
originally introduced by Bonami and Grellier \cite{Bonami2012},
an important work is that Ky \cite{Ky2013} found a suitable subspace $H_{b}^{1}(\rn)$ of $H^1(\rn)$ and
established that $[b,I_\alpha]$ have boundedness from $H_{b}^{1}(\rn)$ to $L^{\frac{n}{n-\alpha}}(\rn)$.
Moreover, Ky \cite{Ky2013} also proved that the $[b,I_\alpha]$ maps $H^1(\rn)$ continuously into the weak Lebesgue space $L^{\frac{n}{n-\alpha}, \, \infty}(\rn)$.
For more related results,
we refer to the interested reader to \cite{Adams1996,D-L-Z2002, HuyKy2021, A12, Landkof1972, Olsen1995, Stein1970,Stampacchia1958}.

Motivated by the above works, the present paper studies the
 boundedness of commutator generated by Orlicz-$\mathrm{BMO}$ function and fractional integral operator within the framework of Orlicz-type spaces. To formulate our main results precisely, it is necessary to establish some definitions and notation.

A non-decreasing function $\varphi: [0,\infty)\rightarrow [0,\infty)$ is termed Orlicz function if it satisfies: $\varphi(0)=0$, $\varphi(s)>0$ for all $s>0$ and $lim_{s\rightarrow\infty}\varphi(s)= \infty$. For Orlicz function $\varphi$, it is of lower type $p_1$ (with $p_1\in [0,\infty)$) whenever one can find a constant $C_{p_1}>1$ such that
\begin{align*}
\varphi(st)\leq C_{p_1} s^{p_1} \varphi(t) \qquad \forall\, t\geq0,\; s\in(0,1].
\end{align*}
The symbol
\begin{align*}
 i(\varphi):=\sup\{p_1\geq 0: \varphi\,\text{is of lower type}\, p_1\}
\end{align*}
represents the critical lower type index of $\varphi$.
For Orlicz function $\varphi$, it is of upper type $p_2$ (with $p_2\in [0,\infty)$) whenever one can find a constant $C_{p_2}>1$ such that
\begin{align*}
\varphi(st)\leq C_{p_2}s^{p_2}\varphi(t)\qquad \forall\, t\geq0,\,\; s\in [1,\infty).
\end{align*}
The symbol
\begin{align*}
  I(\varphi):=\inf\{p_2\geq 0: \varphi\,\text{is of upper type}\,p_2
 \}
\end{align*}
represents the critical upper type index of $\varphi$.
\begin{definition}
We term $\varphi: [0,\infty)\to [0,\infty)$ a growth function if it is of upper type $1$ and lower type $p$ with $p\in(0,1]$.
\end{definition}

\begin{definition}[\cite{Ky2014}]\label{def.1.1}
For a given Orlicz function $\varphi$, the space $\mathrm{BMO}_{\varphi}(\mathbb{R}^n)$ stands for the set of all locally integrable functions
$f$ on $\mathbb{R}^n$ satisfying
 \begin{align*}
 \|g\|_{\mathrm{BMO}_{\varphi}(\mathbb{R}^n)}:= \sup_{B \subset \mathbb{R}^n}\frac{1}{\|\mathbf{1}_B\|_{L^{\varphi}(\mathbb{R}^n)}}\int_B |g(x)-g_B|\,dx <\infty,
 \end{align*}
where the supremum ranges over all balls $B$ on $\mathbb{R}^n$ and
 $$g_B :=\frac{1}{|B|}\int_B g(x)\, dx.$$
If $\varphi(t)=t$, $\mathrm{BMO}_{\varphi}(\mathbb{R}^n)=\mathrm{BMO}(\mathbb{R}^n)$.
For more research on $\mathrm{BMO}$ type spaces, see \cite{Adams1996, F-L2023, F-L JFA2024, FLZ25, Y2017}.
\end{definition}
The Orlicz-Lebesgue space $L^\varphi(\rn)$ contains all locally integrable functions $g$ for which
$$\|g\|_{L^\varphi(\rn)}:=\inf\left\{\lambda>0: \, \int_{\rn}\varphi(|g(x)|/\lambda)\,dx\leq 1\right\}<\infty.$$
We denote by $\mathcal{S}(\mathbb{R}^n)$ the Schwartz class of rapidly decreasing smooth functions on $\mathbb{R}^n$, and by $\mathcal{S}'(\mathbb{R}^n)$
its topological dual
(space of tempered distributions). Define
$$\mathcal{S}_{0}(\mathbb{R}^n):=\left\{\psi\in \mathcal{S}(\mathbb{R}^n) : \,
\|\psi\|= \sup_{\substack{x \in \mathbb{R}^n \\ |\alpha|\leq 1}} (1+|x|)^{2(n+1)} \, |\partial_x^\alpha \psi(x)| \leq 1 \right\}.$$
Let $g\in\mathcal{S}'(\mathbb{R}^n)$ and $\psi_t(\cdot) = t^{-n} \psi(\cdot/t)$.
The nontangential grand maximal function $\mathfrak{M}g$ is defined by
\begin{align*}
\mathfrak{M}(g)(x):= \sup_{\psi \in S_0(\mathbb{R}^n)} \; \sup_{\substack{y \in \mathbb{R}^n \\ |y-x| < t}} |(g * \psi_t)(y)|  \quad \text{for all}\, x\in\mathbb{R}^n.
\end{align*}

\begin{definition}
The Orlicz-Hardy space $H^\varphi(\mathbb{R}^n)$ is the collection of all tempered distributions
 $g$ such that $\mathfrak{M}g$ belongs to $L^\varphi(\mathbb{R}^n)$. Its (quasi-)norm is
\begin{align*}
\|g\|_{H^\varphi}:= \|\mathfrak{M}g\|_{L^\varphi}.
\end{align*}
\end{definition}

 \begin{definition}
Suppose that $\varphi$ is a growth function satisfying
 $$\frac{n}{n+1}<i(\varphi)\leq1.$$
Given $b\in\mathrm{BMO}_{\varphi}(\rn)$,
we denote by $H_b^{\varphi}(\rn)$ the collection of all functions $g\in H^{\varphi}(\rn)$ for which the function
$$[b,\mathfrak{M}]g:=\mathfrak{M}\left((b-b(\cdot))g(\cdot)\right)$$
is an element of $L^1(\mathbb{R}^n)$. For any $g\in H_b^{\varphi}(\rn)$, define
\begin{align}\label{intro1.2}
\|g\|_{H_b^\varphi(\mathbb{R}^n)}:= \|g\|_{H^\varphi(\mathbb{R}^n)} \|b\|_{\mathrm{BMO}_\varphi(\mathbb{R}^n)} + \|[b,\mathfrak{M}]\|_{L^1(\mathbb{R}^n)}.
\end{align}
\end{definition}


The subsequent primary results offer a generalisation of the preceding results on the boundedness in Ky \cite{Ky2013} to the Orlicz setting.

\begin{theorem}\label{Theo-main-b-1}
Suppose that $0<\alpha<n$ and $\varphi$ is growth function with
\begin{align}\label{1.2}
n/(n+1)<i(\varphi)\leq I(\varphi)<1
\end{align}
and
\begin{align}\label{eq.1.5}
\begin{cases}
1/I(\varphi)>1/i(\varphi)-1 \quad  & \mbox{if}\ n=1;\\
1/I(\varphi)>(\lfloor n(1/i(\varphi)-1)\rfloor+n-1)/n \quad  & \mbox{if}\ n\geq 2.
\end{cases}
\end{align}
If $b\in\mathrm{BMO}_{\varphi}(\rn)$ and $g\in H_b^{\varphi}(\rn)$, then there exists a positive constant $C$  such that
$$\|[b, I_\alpha]g\|_{L^{\frac{n}{n-\alpha}}(\rn)}\leq C \|g\|_{H_b^{\varphi}(\rn)}.$$
\end{theorem}

\begin{theorem}\label{Theo-main-2}
Under the same assumptions on $\varphi$, $\alpha$, and $b$ as in Theorem \ref{Theo-main-b-1}, if $g\in H^{\varphi}(\rn)$,
then there exists constant $C>0$ such that
$$\|[b, I_\alpha]g\|_{L^{\frac{n}{n-\alpha},\,\infty}(\rn)}\leq C\|g\|_{H^{\varphi}(\rn)}.$$
\end{theorem}

\begin{remark}\rm Theorems \ref{Theo-main-b-1} and \ref{Theo-main-2} are a further extension for the results of Ky \cite[Corollary 8.1 and Theorem 8.2]{Ky2013} in the Orlicz-type spaces framework. A study related to the Calder\'on-Zygmund operator, see \cite{F-L2024}.
\end{remark}

In Section \ref{s2}, we introduce the operator class $\mathcal{K}_{\varphi, \, \alpha}$ and prove $I_\alpha\in \mathcal{K}_{\varphi, \, \alpha}$. Then we establish the decomposition for commutator via wavelets in Section \ref{s3}. Finally, Theorems \ref{Theo-main-b-1} and \ref{Theo-main-2} are verified in Section \ref{s4}.

Finally, we give the notational rules adopted in this work.
Let $\mathbb{N}:=\{1,2,\cdots\}$.
For every $ t\in[0,\infty)$,
$\lfloor t\rfloor$ denotes the supremum of all integers $k$ with $k\leq t$.
For $x_{0}\in\rn$ and $r\in (0,\infty)$, define $B(x_0, r):= \{z\in\rn: \, |z-x_0|<r\}$.
The symbol $g_{1}\ls g_{2}$ (or $g_{2}\gs g_{1})$ means $g_{1}\leq Cg_{2}$ (or $g_{2}\geq Cg_{1}$), and $g_{1}\approx g_{2}$ means both $g_{1}\ls g_{2}$ and $g_{2}\ls g_{1}$ hold.

\section{Class $\mathcal{K}_{\varphi, \, \alpha}$ and typical example}\label{s2}
In order to give the definition of $\mathcal{K}_{\varphi, \, \alpha}$,
we start by recalling from \cite{Ky2014} the atomic description of Orlicz-Hardy spaces.
Subsequently, a typical example is presented in $\mathcal{K}_{\varphi, \, \alpha}$.

For ball $B\subset\mathbb{R}^n$ and $q\in[1,\infty]$, the symbol $L^q_{B}$ means the space of all measurable functions $g$ on $\mathbb{R}^n$, which $g$ satisfies $\operatorname{supp}g\subset B$, the norm of $L^q_{B}$ is defined by
\begin{align}\label{L_B^{q}666}
\|g\|_{L^q_{B}}:=
\begin{cases}
\displaystyle\left(\frac{1}{|B|}\int_B |g(x)|^q\,dx\right)^{\!1/q}<\infty, & 1\le q<\infty;\\[2.5ex]
\displaystyle\operatorname{ess\,sup}_{x\in B}|g(x)|<\infty, & q=\infty.
\end{cases}
\end{align}

\begin{definition}\label{atom}
Suppose that $\varphi$ is growth function with the critical lower type exponent $i(\varphi)$.
A triple $(\varphi, q, s)$ is termed admissible if $q\in(1,\infty]$
and the positive integer $s$ satisfies
$$s\geq m(\varphi):=\lfloor n(1/i(\varphi)-1)\rfloor.$$
A locally integrable function $a$ on $\rn$ is termed $(\varphi, q, s)$-atom if the next conditions are valid:
\begin{enumerate}
\item[\rm (i)] $\operatorname{supp} a \subset B$ for some ball $B\subset\mathbb{R}^n$;
\item[\rm (ii)] $\left\|a\right\|_{L^{q}_{B}}\leq\left\|{\mathbf 1}_{B}\right\|^{-1}_{L^{\varphi}(\rn)}$;
\item[\rm (iii)] $\int_{\mathbb{R}^{n}}x^{\gamma}a(x)\,dx=0$  for every multi-indices $\gamma\in\zz_+^n$ with $|\gamma|\leq s$.
\end{enumerate}
\end{definition}

\begin{definition}
For $1\le q\le\infty$ and $s\geq m(\varphi):=\lfloor n(1/i(\varphi)-1)\rfloor$,  suppose that $\{b_j\}_{j}$ is a sequence of multiples of $(\varphi,q,s)$-atom, each of which is supported in ball $B_j$. Define
$$\Lambda_q\left(\{b_j\}\right):=\inf\left\{\lambda>0:\sum_{j}\varphi\left(\frac{\|b_j\|_{L_{B_j}^q}}{\lambda}\right)|B_j|\leq 1\right\}.$$
\end{definition}

\begin{definition}\label{def:class-K}
Let $0<\alpha<n$ and $\varphi$ be growth function such that
$n/(n+1)<i(\varphi)\le 1$.
The $\mathcal{K}_{\varphi, \, \alpha}$ stands for the set consisting of all sublinear operator $R$ that fulfills the following three conditions below:
\begin{enumerate}
\item[\rm (i)] $R$ is bounded from $H^{1}(\mathbb{R}^{n})$ to $L^{\frac{n}{n-\alpha}}(\rn)$;
\item[\rm (ii)]$R$  is bounded from $L^{1}(\mathbb{R}^{n})$ to
$L^{\frac{n}{n-\alpha}, \, \infty}(\rn)$;
\item[\rm (iii)] For every $b\in\mathrm{BMO}_{\varphi}(\mathbb{R}^{n})$ and every $(\varphi,q,0)$-atom $a$ whose support is contained in a ball $B\subset\mathbb{R}^{n}$ with $q\in(1,n/\alpha)$, there exists constant $C>0$ such that
\begin{align*}
\|(b-b_{B})Ra\|_{L^{\frac{n}{n-\alpha}}(\mathbb{R}^{n})}
\leq C\|a\|_{L^{q}_{B}}\|{\mathbf{1}_{B}\|_{L^{\varphi}(\rn)}\|b\|_{\mathrm{BMO}_{\varphi}(\rn)}}.
\end{align*}
\end{enumerate}
\end{definition}

To facilitate subsequent proof, it is proposed that the result in \cite[(2.3)]{F-L2024} be employed as a lemma.

\begin{lemma}\label{lemma2.3}
Let $\varphi: [0,\infty)\to [0,\infty)$ is of upper type $p_{1}\in(0,1]$ and lower type $p_{2}\in(0,1]$. Then, for ball $B(x_0, r)$ and $t>0$,
\begin{align*}
\|{\mathbf 1}_{B(x_0,\, tr)}\|_{L^{\varphi}(\rn)}\leq
\begin{cases}
C^{+}t^{n/p_{1}}\|{\mathbf 1}_{B(x_0,\, r)}\|_{L^{\varphi}(\rn)}  &\text{as}\  t\in(0, 1];\\
C^{-}t^{n/p_{2}}\|{\mathbf 1}_{B(x_0,\, r)}\|_{L^{\varphi}(\rn)}  &\text{as}\ t\in[1, \infty),
\end{cases}
\end{align*}
where $C^{+}$ and $C^{-}$ are positive constants.
\end{lemma}

\begin{lemma}\label{lem1.2}
Let $0<\alpha<n$ and $\varphi$ be growth function with
$n/(n+1)<i(\varphi)\leq 1$. Then $I_{\alpha}\in\mathcal{K_{\varphi,\,\alpha}}$.
\end{lemma}

\begin{proof}
Accoording to\cite[Theorem~3.1]{H^p Lu} and \cite[Theorem~1.2.3]{GTM250}, $I_{\alpha}$ has boundedness from $H^{1}(\mathbb{R}^{n})$ to $L^{\frac{n}{n-\alpha}}(\rn)$, and $I_\alpha$ also bounded from $L^{1}(\mathbb{R}^{n})$ to $L^{\frac{n}{n-\alpha}, \, \infty}(\rn)$.
In order to prove this lemma, it suffices to show that $I_{\alpha}$ satisfies Definition \ref{def:class-K}(iii).

Now, we begin verify that $I_{\alpha}$ satisfies Definition \ref{def:class-K}(iii).
From the condition $n/(n+1)<i(\varphi)$,
we should choose $\sigma$ such that $\varphi$ is of lower type $\sigma$ with $n/(n+1)< \sigma < i(\varphi)$. So,
\begin{align}\label{nsigma-1}
n\left(\frac{1}{\sigma}-1\right)<1.
\end{align}
Suppose $b\in\mathrm{BMO}_{\varphi}(\rn)$ and  $a$ is $(\varphi, q, 0)$ atom
with $\operatorname{supp}a \subset B$.
Denote
$$p_\alpha = \frac{n}{n-\alpha}.$$
Write
\begin{align}\label{I1-II2}
\|(b-b_{B})I_\alpha a\|_{L^{\frac{n}{n-\alpha}}(\mathbb{R}^{n})}
&\leq\|(b-b_B)\mathbf{1}_{2B} I_\alpha a\|_{L^{p_\alpha}(\rn)}+ \|(b-b_B)\mathbf{1}_{(2B)^c} I_\alpha a\|_{L^{p_\alpha}(\rn)}\\
&=:{\rm I}+{\rm II}.\notag
\end{align}

To estimate ${\rm I}$, if $q\in(1, n/\alpha)$ and $1/q_0=1/q-\alpha/n$,
then we see that $I_{\alpha}$ is bounded from $L^{q}(\mathbb{R}^{n})$ to $L^{q_0}(\mathbb{R}^{n})$.
Taking $r>1$ with $1/r+1/q_0=1/p_\alpha$,
by H\"{o}lder inequality, Lemma \ref{lemma2.3}
and \cite[Lemma~3.4 and Lemma~5.5]{F-L2024}, it follows that
\begin{align}\label{eq-I1}
{\rm I}
&\leq \left( \left( \int_{2B} \left| b(x) - b_B \right|^{p_\alpha \cdot \frac{r}{p_\alpha}}\,dx \right)^{\frac{p_\alpha}{r}} \cdot \left( \int_{2B}
\left| I_\alpha a(x) \right|^{p_\alpha \cdot \frac{q_0}{p_\alpha}}\,dx \right)^{\frac{p_\alpha}{q_0}} \right)^{\frac{1}{p_\alpha}}\\
&= \left( \int_{2B} |b(x) - b_B|^r\,dx \right)^{\frac{1}{r}} \left(\int_{2B} |I_\alpha a(x)|^{q_0}\,dx \right)^{\frac{1}{q_0}}\notag \\
&\ls \|a\|_{L^q(\mathbb{R}^n)}\left( \left( \int_{2B} |b(x) - b_{2B}|^r\,dx \right)^{\frac{1}{r}} + \left( \int_{2B} |b_{2B} - b_B|^r\,dx \right)^{\frac{1}{r}}\right)\notag \\
&\ls \|a\|_{L^{q}(\mathbb{R}^n)}\left(|2B|^{\frac{1}{r}-1}\|{\mathbf 1}_{2B}\|_{L^{\varphi}(\mathbb{R}^n)} \|b\|_{\mathrm{BMO}_{\varphi}(\mathbb{R}^n)} + |2B|^{\frac{1}{r}}2^{n/\sigma}|2B|^{-1}\|b\|_{\mathrm{BMO}_{\varphi}}\|{\mathbf 1}_{B}\|_{L^{\varphi}(\mathbb{R}^n)}\right)\notag\\
&\ls\|a\|_{L^q(\mathbb{R}^n)} \ |B|^{\frac{1}{r}-1} \|{\mathbf{1}}_{B}\|_{L^{\varphi}(\mathbb{R}^n)} \|b\|_{\mathrm{BMO}_{\varphi}(\mathbb{R}^n)}\notag\\
&= \|a\|_{L^{q}_{B}} |B|^{\frac{1}{q}+\frac{1}{r}-1} \|\mathbf{1}_B\|_{L^\varphi(\mathbb{R}^n)}\|b\|_{\mathrm{BMO}_{\varphi}(\mathbb{R}^n)}\notag\\
&= \|a\|_{L^{q}_{B}} \|\mathbf{1}_B\|_{L^\varphi(\mathbb{R}^n)}\|b\|_{\mathrm{BMO}_{\varphi}(\mathbb{R}^n)},\notag
\end{align}
where the penultimate step used $\|a\|_{L^q(\rn)}=\|a\|_{L^{q}_{B}}|B|^{1/q}$ (see \eqref{L_B^{q}666}) and the last step used
$$\frac{1}{q}+\frac{1}{r}-1=\frac{1}{q_0}+\frac{\alpha}{n}+\frac{1}{p_\alpha}-\frac{1}{q_0}-1
=\frac{\alpha}{n}+\frac{1}{p_\alpha}-1=0.$$

Below, we consider  ${\rm II}$.
Denote $B=B(x_0, r)$. Due to the fact that $a$ is a $(\varphi, q, 0)$-atom supported on a ball $B\subset\rn$, we have $\operatorname{supp}a\subset B$ and $\int_{\rn}a(x)\,dx=0$.
For any $z\in B$ and $x\in(2^{j+1}B)\setminus  (2^{j}B)$ with $j\in\nn$,
we find a point, which is labeled by $\xi_z$, lying on the straight segment joining $z$ to $x_0$, satisfying that
\begin{align}\label{eq.1.2}
|I_\alpha a(x)|&
=\left|\int_B a(z)\left(\frac{1}{|x-z|^{n-\alpha}}-\frac{1}{|x-x_0|^{n-\alpha}}\right)\,dz\right|\\
&\leq\int_B\left|\frac{1}{|x-z|^{n-\alpha}}-\frac{1}{|x-x_0|^{n-\alpha}}\right||a(z)|\,dz\notag\\
&\leq  \int_B |z-x_0|\left|(n-\alpha)|x-\xi_z|^{-(n-\alpha+1)}\right||a(z)|\,dz\notag\\
&\leq (n-\alpha) r \cdot \lf(\frac{2^{j}r}{2}\r)^{-(n-\alpha+1)}
\int_B |a(z)|\,dz,\notag
\end{align}
where the last step utilized
$$|x-\xi_z|\geq|x-x_0|-|x_{0}-\xi_z|\geq |x-x_0|-\frac{|x-x_0|}{2}=\frac{|x-x_0|}{2}\geq \frac{2^{j}r}{2}.$$
In addition, by $\|a\|_{L^q(\rn)}=\|a\|_{L^{q}_{B}}|B|^{1-1/q'}$,  we get
\begin{align*}
\int_B |a(z)|\,dz
\leq |B|^{1/q'} \|a\|_{L^{q}}
=|B|\|a\|_{L^{q}_{B}},
\end{align*}
which and \eqref{eq.1.2} imply
\begin{align*}
|I_\alpha a(x)|
\ls r\cdot r^{n}\cdot r^{-(n-\alpha+1)}\cdot 2^{-j(n-\alpha+1)}\|a\|_{L^{q}_{B}}
= r^{\alpha}\cdot 2^{-j(n-\alpha+1)}\|a\|_{L^{q}_{B}}.
\end{align*}
This, along with lemma \ref{lemma2.3} and \cite[Lemma~3.4 and Lemma~5.5]{F-L2024}, yields
\begin{align*}
{\rm II}
&=\left(\int_{(2B)^c} \left|I_\alpha a(x)\left(b(x)-b_B\right)\right|^{p_\alpha}\,dx \right)^{\frac{1}{p_\alpha}}\\
&\leq\sum_{j=1}^{\infty}\left(\int_{(2^{j+1}B)\setminus(2^{j}B)}\left|I_\alpha a(x) \left(b(x)-b_B\right)\right|^{p_\alpha}\,dx\right)^{\frac{1}{p_\alpha}}\\
&\ls r^{\alpha}\|a\|_{L_B^{q}}\sum_{j=1}^{\infty}2^{-j(n-\alpha+1)}
\Bigg[
    \left(\int_{2^{j+1}B}\left|b(x)-b_{2^{j+1}B}\right|^{p_\alpha}\,dx \right)^{\frac{1}{p_\alpha}}\\
&\qquad +\left(\int_{2^{j+1}B}\left|b_B-b_{2^{j+1}B}\right|^{p_\alpha}\,dx \right)^{\frac{1}{p_\alpha}}\Bigg]\\
&\ls r^{\alpha}\|a\|_{L_B^{q}}\sum_{j=1}^{\infty}2^{-j(n-\alpha+1)}
\Big(|2^{j+1}B|^{\frac{1}{p_\alpha}-1}\|{\mathbf 1}_{2^{j+1}B}\|_{L^{\varphi}(\rn)}\|b\|_{\mathrm{BMO}_{\varphi}(\rn)}\\
&\qquad+|2^{j+1}B|^{\frac{1}{p_\alpha}}2^{\frac{(j+1)n}{\sigma}}|2^{j+1}B|^{-1}
\|\mathbf{1}_{B}\|_{L^{\varphi}(\mathbb{R}^n)} \|b\|_{\mathrm{BMO}_{\varphi}(\rn)}\Big)\\
&\ls r^{\alpha} \|a\|_{L^{q}_{B}}
\sum_{j=1}^{\infty} 2^{-j(n-\alpha+1)} |2^{j+1}B|^{\frac{1}{p_\alpha}-1}
2^{\frac{(j+1)n}{\sigma}} \|\mathbf{1}_B\|_{L^{\varphi}(\rn)} \|b\|_{\mathrm{BMO}_{\varphi}(\rn)}\\
&\approx r^{\alpha}r^{n\lf(\frac{1}{p_\alpha}-1\r)}\|a\|_{L^{q}_{B}}
\|\mathbf{1}_B\|_{L^{\varphi}(\rn)} \|b\|_{\mathrm{BMO}_{\varphi}(\rn)}\left(\sum_{j=1}^{\infty}
2^{-j(1+n-\alpha)+n(j+1)(\frac{1}{p_\alpha}-1)+\frac{(j+1)n}{\sigma}}\right).
\end{align*}
Further, by $\alpha=-n\lf(1/p_\alpha-1\r)$ and
$$-j(1+n-\alpha)+n(j+1)\left(\frac{1}{p_\alpha}-1\right)+\frac{(j+1)n}{\sigma}
=-j\lf(1-n\left(\frac{1}{\sigma} - 1\right)\r)+\frac{n}{\sigma}-\alpha,$$
we use \eqref{nsigma-1} to obtain that
\begin{align}\label{eq-II1}
\mathrm {II}\ls \mathit{\|a\|_{L^{q}_{B}}\|\mathbf{1}_B\|_{L^{\varphi}(\rn)} \|b\|_{\mathrm{BMO}_{\varphi}(\rn)}}\sum_{j=1}^{\infty}
2^{-j\lf(1-n\left(\frac{1}{\sigma} - 1\right)\r)}
\ls \mathit{\|a\|_{L^{q}_{B}}\|\mathbf{1}_B\|_{L^{\varphi}(\rn)} \|b\|_{\mathrm{BMO}_{\varphi}(\rn)}}.
\end{align}
Substituting \eqref{eq-I1} and \eqref{eq-II1} into \eqref{I1-II2} yields that $I_\alpha$ satisfies Definition \ref{def:class-K}(iii).
\end{proof}

\section{Decomposition of commutator}\label{s3}

This section is devoted to find a wavelet decomposition of commutator,
where the wavelet basis is introduced in \cite{Wavelet 1992}.

Denote $F=\{0,1\}^n\setminus\{(0,\dots,0)\}$, the function of all dyadic cubes in $\mathbb{R}^n$ is written as $\mathcal{D}$.  Let $\{\psi_I^\lambda\}_{I\in\mathcal{D},\lambda\in F}$ be a normalized and orthogonal wavelet basis of $L^2(\mathbb{R}^n)$ with compact support and sufficient smoothness (see \cite{Wavelet 1992}).

Any $g\in L^2(\mathbb{R}^n)$ can be expanded as
$$g=\sum_{I\in\mathcal{D}}\sum_{\lambda\in F}\langle g,\psi_I^\lambda\rangle\psi_I^\lambda \ \quad \text{for all} \ g\in L^2(\mathbb{R}^n).
$$

According to \cite{Bonami2012}, we define the bilinear operators $\Pi_i\ (i=1,2,3,4)$ by
\begin{align}\label{3.9}
\Pi_1(g_1,g_2)=\sum_{\substack{I,I'\in\mathcal{D}\\|I|=|I'|}}\sum_{\lambda\in F}\langle g_1,\varphi_I\rangle\langle g_2,\psi_{I'}^\lambda\rangle\varphi_I\psi_{I'}^\lambda,
\end{align}
\begin{align}\label{3.10}
\Pi_2(g_1,g_2)=\sum_{\substack{I,I'\in\mathcal{D}\\|I|=|I'|}}\sum_{\lambda\in F}\langle g_1,\psi_I^\lambda\rangle\langle g_2,\varphi_{I'}\rangle\psi_I^\lambda\varphi_{I'},
\end{align}
\begin{align}\label{3.11}
\Pi_3(g_1,g_2)=\sum_{\substack{I,I'\in\mathcal{D}\\|I|=|I'|}}\sum_{\substack{\lambda,\lambda'\in F\\(I,\lambda)\neq(I',\lambda')}}\langle g_1,\psi_I^\lambda\rangle\langle g_2,\psi_{I'}^{\lambda'}\rangle\psi_I^\lambda\psi_{I'}^{\lambda'}
\end{align}
and
\begin{align}\label{3.12}
\Pi_4(g_1,g_2)=\sum_{I\in\mathcal{D}}\sum_{\lambda\in F}\langle g_1,\psi_I^\lambda\rangle\langle g_2,\psi_I^\lambda\rangle(\psi_I^\lambda)^2.
\end{align}

\begin{proposition}\label{prop:3.14}
Suppose that $\varphi$ is a growth function, its lower and upper critical indices satisfy
\eqref{1.2}.
Respectively, each $\Pi_i$ ($i=1,2,3,4$) corresponds to the expression in \eqref{3.9}, \eqref{3.10}, \eqref{3.11} and \eqref{3.12}.
Consequently, the statements below are true.
\begin{enumerate}
\item[\rm (i)] Both $\Pi_1$ and $\Pi_3$ are bounded from
$H^{\varphi}(\mathbb{R}^n) \times \mathrm{BMO}_{\varphi}(\mathbb{R}^n)$ to
$H^1(\mathbb{R}^n)$;

\item[\rm (ii)] $\Pi_4$ is bounded from $H^{\varphi}(\mathbb{R}^n) \times \mathrm{BMO}_{\varphi}(\mathbb{R}^n)$ to $L^1(\mathbb{R}^n)$;

\item[\rm (iii)] Suppose that $I(\varphi)\in \left(n/(n+1),1\right)$ and \eqref{eq.1.5} holds for $\varphi$, a sequence of dyadic cubes are denoted by $R_j$. Given $b\in\mathrm{BMO}_\varphi$ and a finite atomic decomposition
$f=\sum_{j=1}^{M}a_j$ with $f\in H^{\varphi}(\rn)$, where each $a_j$ can be written as a constant times a $(\varphi,\infty,s)$ atom, and $\operatorname{supp}a_j\subset R_j$.
Then
$$\sum_{j=1}^{M} \big\| \Pi_2(a_j, b - b_{B_j}) \big\|_{H^1(\mathbb{R}^n)}
\;\lesssim\; \|f\|_{H^\varphi(\mathbb{R}^n)} \, \|b\|_{\mathrm{BMO}_\varphi(\mathbb{R}^n)},$$
where $b_{B_j}$ is the average of $b$ over ball $B_j$ and $B_j\supseteq R_j$ is the smallest ball.
\end{enumerate}
\end{proposition}
\begin{proof}
Note that (i) and (ii) follow from \cite[Propositions~3.14, 3.16 and 3.17]{F-L2024}.
By \cite[(5.17)]{F-L2024}, we obtain that (iii) holds.
\end{proof}


Next we will establish a decomposition for the commutator $[b,T]$. The proof relies on a wavelet orthonormal basis $\{\psi_l^\lambda\}$ (see \cite{Wavelet 1992}) and the associated bilinear operators $\Pi_i$ introduced in \cite{Bonami 2019}.

\begin{lemma}\label{1.4}
Given a growth function $\varphi$ and its lower and upper critical indices satisfy \eqref{1.2} and \eqref{eq.1.5}.
Let $\Pi_{4}$ be defined as \eqref{3.12}
and $T\in \mathcal{K}_{\varphi,\,\alpha}$ with $0<\alpha<n$.
For any $b\in\mathrm{BMO}_{\varphi}(\mathbb{R}^n)$ and any $g\in H^{\varphi}(\mathbb{R}^n)$, there exists an operator
$$\mathfrak{T}:H^{\varphi}(\rn)\times\mathrm{BMO}_{\varphi}(\rn)\to L^{\frac{n}{n-\alpha}}(\rn)$$
such that
\begin{align}\label{eq-bTI}
[b,T]g= \mathfrak{T}(g,\,b)-T(\Pi_{4}(g,\,b)).
\end{align}
\end{lemma}
\begin{proof}
Based on the argument of  \cite[Proposition~6.7]{Wavelet 2018},
the following three steps will be used to prove the desired result.

\medskip

{\it Step 1: Verifying \eqref{eq-bTI} under the case that $g$ has a finite linear combination of wavelets.}
We unitize \cite[Lemma~3.3, (c) and (a)]{F-L2024} to obtain
that $g$ has finite atomic decomposition
$$g=\sum^{M}_{k=1}a_{k}$$
with
$$a_k=\sum_{\substack{I\in\mathcal{D}\\ I\subset R_k}}\sum_{\lambda\in F}
\langle g,\psi^{\lambda}_{I}\rangle\psi^{\lambda}_{I}\qquad (k=1,\dots,M)$$
and
$$\Lambda_{\infty}\bigl(\{a_k\}_{k=1}^M\bigr)\lesssim\|g\|_{H^{\varphi}(\mathbb{R}^n)},$$
where $\operatorname{supp} a_k \subseteq R_k$
and $\langle g,\psi^{\lambda}_{I}\rangle$ is the coefficients of $a_k$.
Further, let $B_k$ be the smallest ball containing dyadic cube $R_k$.
According to \cite[Proposition~2.5]{F-L2024}, we have
\begin{equation}\label{eq.1.7}
\sum_{k=1}^{M}\|a_k\|_{L^{\infty}_{B_k}}\|\mathbf{1}_{B_k}\|_{L^{\varphi}(\mathbb{R}^n)}
\lesssim\Lambda_{\infty}\bigl(\{a_k\}_{k=1}^M\bigr)
\lesssim\|g\|_{H^{\varphi}(\mathbb{R}^n)}.
\end{equation}

Denote
$$\mathcal{W}(g,\,b):=T\left(\sum_{k=1}^{M}\left[(b-b_{B_k})a_k
-\Pi_{2}\left(a_k,b-b_{B_k}\right)\right]\right).$$
For almost every $x\in\mathbb{R}^n$, using a standard algebraic identity (see \cite[(6.9)]{Wavelet 2018}), it holds
\begin{align*}
b(x)g(\cdot)-b(\cdot)g(\cdot)=&\sum_{k=1}^{M}\Bigl[\bigl(b(x)-b_{B_k}\bigr)a_k(\cdot)
-\Pi_{2}\bigl(a_k,b-b_{B_k}\bigr)(\cdot)\Bigr]\\ \notag
&\quad-\Pi_{1}(g,b)(\cdot)-\Pi_{3}(g,b)(\cdot)-\Pi_{4}(g,b)(\cdot).
\end{align*}
This implies
\begin{align}\label{eq-3.13}
[b,T]g= \mathcal{W}(g,\,b)-T(\Pi_1(g,\,b))-T(\Pi_3(g,\,b))-T(\Pi_4(g,\,b)).
\end{align}
Now, if we define
$$\mathfrak{T}(g,\,b)= \mathcal{W}(g,\,b)-T(\Pi_1(g,\,b))-T(\Pi_3(g,\,b)),$$
then immediately obtain
$$[b,T]g=\mathfrak{T}(g,\,b)-T(\Pi_4(g,\,b)),$$
which is the desired equality.

\medskip
{\it Step 2:
For $b\in\mathrm{BMO}_{\varphi}(\rn)$ and $g\in H^{\varphi}(\rn)$ with finite linear combination of wavelets, it follows that
$$\|\mathfrak{T}(g,\,b)\|_{L^{\frac{n}{n-\alpha}}(\rn)}\ls \|g\|_{H^{\varphi}(\rn)}\|b\|_{\mathrm{BMO}_{\varphi}(\rn)}.$$}
Using Proposition \ref{prop:3.14},
 for arbitrary $g\in H^{\varphi}(\rn)$ and $b\in\mathrm{BMO}_{\varphi}(\rn)$,
$$\|\Pi_{i}(g,\,b)\|_{H^1(\rn)}\ls \|g\|_{H^{\varphi}(\rn)}\|b\|_{\mathrm{BMO}_{\varphi}(\rn)},\quad i=1, 3.$$
By this, $T\in \mathcal{K}_{\varphi,\,\alpha}$, Proposition \ref{prop:3.14}(iii) , \eqref{eq.1.7} and $\|a_k\|_{L^{q}_{B_k}}\leq\|a_k\|_{L^{\infty}_{B_k}}$ with $q\in(1, n/\alpha)$, we see that
\begin{align*}
\|\mathfrak{T}(g,b)\|_{L^{\frac{n}{n-\alpha}}(\mathbb{R}^n)}
&\le \|\mathcal{U}(g,b)\|_{L^{\frac{n}{n-\alpha}}(\mathbb{R}^n)}+\|T(\Pi_{1}(g,b))\|_{L^{\frac{n}{n-\alpha}}(\mathbb{R}^n)}
+\|T(\Pi_{3}(g,b))\|_{L^{\frac{n}{n-\alpha}}(\mathbb{R}^n)}\\
&\lesssim \sum_{k=1}^{M}\lf\|(b(\cdot)-b_{B_k})Ta_k\r\|_{L^{\frac{n}{n-\alpha}}(\mathbb{R}^n)}
+\lf\|T\lf(\sum_{k=1}^{M}\Pi_{2}(a_k,b-b_{B_k})\r)\r\|_{L^{\frac{n}{n-\alpha}}(\mathbb{R}^n)}\\
&\qquad +\|\Pi_{1}(g,b)\|_{H^{1}(\mathbb{R}^n)}+\|\Pi_{3}(g,b)\|_{H^{1}(\mathbb{R}^n)}\\
&\lesssim \|g\|_{H^{\varphi}(\mathbb{R}^n)}\|b\|_{\mathrm{BMO}_{\varphi}(\mathbb{R}^n)}
+\sum_{k=1}^{M}\|a_k\|_{L^{q}_{B_k}}\|\mathbf{1}_{B_k}\|_{L^{\varphi}(\mathbb{R}^n)}
\|b\|_{\mathrm{BMO}_{\varphi}(\mathbb{R}^n)}\\
&\qquad+\sum_{k=1}^{M}\lf\|\Pi_{2}(a_k,b-b_{B_k})\r\|_{H^{1}(\mathbb{R}^n)}\\
&\lesssim \|g\|_{H^{\varphi}(\mathbb{R}^n)}\|b\|_{\mathrm{BMO}_{\varphi}(\mathbb{R}^n)},
\end{align*}
as desired.

\medskip

{\it Step 3: The proof of \eqref{eq-bTI}.}
Based on the facts obtained in {\it Steps 1} and {\it 2}, we repeat the argument of \cite[(5.18)]{F-L2024} to obtain \eqref{eq-bTI}.
This implies Lemma \ref{1.4}.
\end{proof}

%

\section{Proofs of Theorems \ref{Theo-main-b-1} and \ref{Theo-main-2}}\label{s4}

The primary objective of this section is to complete the proofs of Theorems \ref{Theo-main-b-1} and \ref{Theo-main-2} by using Lemma \ref{Theorem.11}.
We first recall the following lemma from \cite[Theorem~5.10]{F-L2024}.

\begin{lemma}\label{Theorem.11}
Suppose that $\varphi$ is growth function and it satisfies \eqref{1.2} and \eqref{eq.1.5}.
If $b\in\mathrm{BMO}_{\varphi}(\rn)$, then
$g\in H^{\varphi}_{b}(\mathbb{R}^{n})$ is equivalent to $\Pi_{4}(g,\, b)\in H^{1}(\mathbb{R}^{n})$ with
\begin{align*}
\|g\|_{H^{\varphi}_{b}(\mathbb{R}^{n})}
\approx\|g\|_{H^{\varphi}(\mathbb{R}^{n})}\|b\|_{\mathrm{BMO}_{\varphi}(\rn)}+\left\|\Pi_{4}(g,\, b)\right\|_{H^{1}(\mathbb{R}^{n})}.
\end{align*}
\end{lemma}

\begin{proof}[\bf{Proof of Theorem \ref{Theo-main-b-1}}]
Under the assumptions that $b\in\mathrm{BMO}_{\varphi}(\rn)$ and $g\in H^{\varphi}_{b}(\mathbb{R}^{n})$,  by Lemma \ref{Theorem.11},
we have $\Pi_{4}(g,\, b)\in H^{1}(\mathbb{R}^{n})$ and
\begin{align}\label{eq-fh3-4}
\|g\|_{H^{\varphi}_{b}(\mathbb{R}^{n})}
\approx\|g\|_{H^{\varphi}(\mathbb{R}^{n})}
\|b\|_{\mathrm{BMO}_{\varphi}(\rn)}+\|\Pi_{4}(g,\, b)\|_{H^{1}(\mathbb{R}^{n})}.
\end{align}
Based on Lemma \ref{lem1.2}, we use Lemma \ref{1.4} to obtain that there exists operator
$$\mathfrak{T}: H^{\varphi}(\mathbb{R}^{n})\times\mathrm{BMO}_{\varphi}(\rn)\to L^{\frac{n}{n-\alpha}}(\rn)$$
such that
$$[b, I_\alpha]g=\mathfrak{T}(g,\, b)-{I_\alpha}(\Pi_{4}(g,\, b)).$$
From this and \eqref{eq-H-1}, it follows that
\begin{align*}
\|[b, I_\alpha]g\|_{L^{\frac{n}{n-\alpha}}(\rn)}
&=\|\mathfrak{T}(g,\, b)-I_\alpha(\Pi_{4}(g,\, b))\|_{L^{\frac{n}{n-\alpha}}(\rn)}\\
&\leq \|\mathfrak{T}(g,\, b)\|_{L^{\frac{n}{n-\alpha}}(\rn)}+\|I_\alpha(\Pi_{4}(g,\, b))\|_{L^{\frac{n}{n-\alpha}}(\rn)}\\
&\ls\|g\|_{H^{\varphi}(\mathbb{R}^{n})}\|b\|_{\mathrm{BMO}_{\varphi}(\rn)}+\|\Pi_{4}(g,\, b)\|_{H^{1}(\mathbb{R}^{n})}\\
&\approx \|g\|_{H^{\varphi}_{b}(\mathbb{R}^{n})},
\end{align*}
which proves Theorem\ref{Theo-main-b-1}.
\end{proof}

\begin{proof}[\bf{Proof of Theorem \ref{Theo-main-2}}]
By Lemma \ref{1.4}, we find that there exists  operator
$$\mathfrak{T}: H^{\varphi}(\mathbb{R}^{n})\times\mathrm{BMO}_{\varphi}(\rn)\to L^{\frac{n}{n-\alpha}}(\rn)$$
such that for any $b\in\mathrm{BMO}_{\varphi}(\rn)$ and $g\in H^{\varphi}(\mathbb{R}^{n})$,
$$[b, I_\alpha]g(y)=\mathfrak{T}(g,\, b)(y)-{I_\alpha}(\Pi_{4}(g,\, b))(y)
\ \quad\text{for a. e.} \, y\in\rn.$$
So, using \eqref{eq-Ln-L1} and Proposition \ref{prop:3.14}, we have
\begin{align*}
\|[b,I_\alpha]g\|_{L^{\frac{n}{n-\alpha}, \, {\infty}}(\mathbb{R}^n)}
&\ls\|\mathfrak{T}(g,\, b)\|_{L^{\frac{n}{n-\alpha}, \, {\infty}}(\mathbb{R}^n)}+
\|I_\alpha(\Pi_{4}(g,\,b))\|_{L^{\frac{n}{n-\alpha}, \, {\infty}}(\mathbb{R}^n)}\\
&\ls\|\mathfrak{T}(g,\,b)\|_{L^{\frac{n}{n-\alpha}}(\mathbb{R}^n)}+\|\Pi_{4}(g,\,b)\|_{L^{1}(\rn)}\\
&\ls\|g\|_{H^{\varphi}(\mathbb{R}^{n})}\|b\|_{\mathrm{BMO}_{\varphi}(\mathbb{R}^n)}.
\end{align*}
This implies Theorem\ref{Theo-main-b-1}.
\end{proof}

\end{document}